\documentclass[oneside,english]{amsart}
\usepackage[T1]{fontenc}
\usepackage[latin9]{inputenc}
\usepackage{amsthm}
\usepackage{amstext}
\usepackage{amssymb}
\usepackage{esint}
\usepackage{graphicx}
\usepackage{enumerate}
\usepackage{amscd}
\usepackage{color}

\makeatletter
\newtheorem{theorem}{Theorem}[section]
\newtheorem{lemma}[theorem]{Lemma}

\numberwithin{figure}{section}
\theoremstyle{definition}
\newtheorem{definition}[theorem]{Definition}

\theoremstyle{remark}

\numberwithin{equation}{section}
\makeatother

\usepackage{color}
\usepackage{todonotes}

\newcommand{\Om} {\Omega}

\usepackage{babel}

\begin{document}

\title[The nonlinear Steklov problem]
{The nonlinear Steklov problem in outward cuspidal domains}

\author{Pier Domenico Lamberti and Alexander Ukhlov}

\begin{abstract}
In this article, we consider the nonlinear Steklov eigenvalue problem in outward cuspidal domains.
Using the compactness of the weighted trace embedding we obtain the variational characterization of the first non-trivial eigenvalue and prove the existence of a corresponding weak solution.
\end{abstract}

\maketitle



\section{Introduction}

The classical Steklov eigenvalue problem goes back to the original work of Steklov \cite{S02} and has been studied throughout the last century. However, its correct formulation 
in non-Lipschitz domains $\Omega \subset \mathbb{R}^n$ remains a complicated long-standing problem. The main obstacle comes from the fact that Sobolev trace 
operators can be non-compact in non-Lipschitz domains \cite{M,MP}, and in such cases the spectrum can be continuous \cite{NT}. In \cite{NT}, it was proved that for  an outward cuspidal domain with cusp function $\gamma (t)=t^{\alpha}$  the classical  unweighted Steklov problem has a discrete spectrum when $1<\alpha<2$. For $\alpha=2$ the spectrum contains a continuous part, while for $\alpha>2$ the point $\lambda_0=0$ belongs to the continuous spectrum.

In \cite{GGU25}, an approach to the classical Steklov eigenvalue problem in outward $\gamma$-cuspidal domains was proposed, based on introducing weights related to the 
geometry of the domains. In the present work we consider the nonlinear Steklov $p$-eigenvalue problem, $p\ne 2$,
\begin{equation}
\label{WeSP}
\begin{cases}
-\mathrm{div}( |\nabla u|^{p-2}\nabla u ) = 0 & \text{in } \Omega_{\gamma},\\
|\nabla u|^{p-2} \nabla u \cdot \nu = \lambda\, w\, |u|^{p-2}u & \text{on } \partial\Omega_{\gamma},
\end{cases}
\end{equation}
where $\Omega_{\gamma}\subset\mathbb R^n$ is an outward $\gamma$-cuspidal domain and 
the weight $w$ is induced by the cusp function $\gamma$, see \eqref{cusp}.

The nonlinear problems of this type arise in several fields of continuum mechanics, such as fluid mechanics and elasticity (see, e.g., \cite{BS, CPV}), and have attracted growing 
attention in recent years (see, e.g., \cite{B21, BR, CGGS, FL, MN10, XW}).

The linear case ($p=2$) was considered in \cite{GGU25}, where it was proved that the spectrum is discrete. However, in the nonlinear case the corresponding Friedrichs--Poincar\'e type inequality has the critical role. In the present work we establish this inequality and on this base we prove the variational characterization of the first non-trivial eigenvalue $\lambda_p$. Moreover, we show that the orthogonality condition
\[
\int_{\partial \Omega_{\gamma}} |u|^{p-2} u\, w\, ds = 0
\]
implies that the corresponding eigenfunction is non-trivial and that $\lambda_p$ is positive.

The proposed method is based on the compactness of the weighted trace embedding operator
$$
i: W^{1,p}(\Omega_{\gamma})\hookrightarrow L^p_w(\partial\Omega_{\gamma}),\quad 1<p<\infty,
$$
proved in \cite{GV}.

The main result of the article can be stated as follows:

\medskip
\noindent
{\it Let $\Omega_{\gamma}$ be an outward $\gamma$-cuspidal domain of the form \eqref{cusp}. Then for the weighted Steklov $p$-eigenvalue problem \eqref{WeSP}, $1<p<\infty$, there exists 
$u\in W^{1,p}(\Omega_{\gamma})\setminus\{0\}$ such that the first non-trivial eigenvalue $\lambda_{p}$ has the variational characterization
\begin{multline*}
\lambda_{p}=
\inf \left\{\frac{\|\nabla v\|_{L^p(\Omega_{\gamma})}^p}{\|v\|_{L^p_w(\partial \Omega_{\gamma})}^p} : 
v \in W^{1,p}(\Omega_{\gamma}) \setminus \{0\},
\int_{\partial \Omega_{\gamma}} |v|^{p-2}v w\,ds=0 \right\}\\
=\frac{\|\nabla u\|_{L^p(\Omega_{\gamma})}^p}{\|u\|_{L^p_w(\partial \Omega_{\gamma})}^p}.
\end{multline*}
}

\medskip
The variational characterization thus gives the existence of a weak solution to the nonlinear Steklov eigenvalue problem in outward cuspidal domains.

Outward cuspidal geometries have recently received considerable attention in the study of partial differential equations, in connection with geometric singularities and 
weighted functional inequalities; see, for example, \cite{GU09,GU17,KLP,KUZ,MP}. In this setting, the present article establishes a variational characterization of 
nonlinear weighted Steklov eigenvalues in outward cuspidal domains.

\section{Sobolev spaces}

Let us recall the basic notions of the Sobolev spaces. Let $\Omega$ be an open subset of $\mathbb R^n$. The Sobolev space $W^{1,p}(\Omega)$, $1<p<\infty$, is defined \cite{M}
as a Banach space of locally integrable weakly differentiable functions
$u:\Omega\to\mathbb{R}$ equipped with the following norm:
$$
\|u\|_{W^{1,p}(\Omega)}=\left(\int_{\Omega}|\nabla u(x)|^p\,dx+\int_{\Omega}|u(x)|^p\,dx\right)^\frac{1}{p}.
$$

The Sobolev space $W^{1,p}(\Omega)$ arises as the space of weak solutions to partial differential equations~\cite{S36}. In this sense, in contrast to the approach of L.~Schwartz~\cite{LS50},
these spaces can be referred to as Besov--Sobolev spaces~\cite{BIN}.
In this setting, the Sobolev space $W^{1,p}(\Omega)$ coincides with the closure
of the space of smooth functions $C^{\infty}(\Omega)$ in the $W^{1,p}(\Omega)$ norm~\cite{BIN,M}. 

It follows that functions in the Sobolev space $W^{1,p}(\Omega)$ are defined (as elements of a Banach space) up to a set of $p$-capacity zero, in accordance with convergence in the topology of $W^{1,p}(\Omega)$.

The following result, can be found, for example in \cite[Proposition~9.1]{Br11}, \cite[Paragraph~1.4]{IC} and \cite{M}.
\begin{lemma}
\label{Xuthm}
The space $W^{1,p}(\Omega)$, $1<p<\infty$, is real separable and uniformly convex Banach space.
\end{lemma}

Recall the notion of outward $\gamma$-cuspidal domains \cite{GV,KUZ}.  Let $\gamma\in (n,\infty)$ and consider functions $g:[0,1] \to [0,\infty)$ defined by $g(t) = t^{\alpha}$, where $\alpha=(\gamma-1)/(n-1)$.  Denote $x'=(x_1,...,x_{n-1})$. Then an outward $\gamma$-cuspidal domain $\Omega_{\gamma}\subset\mathbb{R}^n$, $n\geq 2$, is defined by
\begin{multline}
\label{cusp}
\Omega_{\gamma}=\\
\left\{(x',x_n)\in \mathbb R^{n-1}\times \mathbb R : \sqrt{x_1^2+...+x_{n-1}^2}<g(x_n), 0<x_n\leq 1\right\}\cup B^n\left((0,2), \sqrt{2}\right),
\end{multline}
where $B^n\left((0,2), \sqrt{2}\right)\subset\mathbb{R}^n$ is the open ball of radius $\sqrt{2}$ centered at $(0,2)\in\mathbb{R}^{n-1}\times\mathbb{R}$.

Note that the intrinsic geometry of outward cuspidal domains in $\mathbb R^n$ is characterized by the parameter $\gamma \in (n,\infty)$ \cite{GU09,KUZ}, while the exponent $\alpha = (\gamma - 1)/(n - 1)$ appears naturally in the power-type cusp function $g(t)=t^{\alpha}$. We follow this standard notation.

Let $E\subset\mathbb R^n$ be a Borel set. Then $E$ is said to be $H^{m}$-rectifiable set \cite{Fe69}, if $E$ is of Hausdorff dimension $m$, and there exists a countable collection $\{\varphi_i\}_{i\in\mathbb{N}}$ of Lipschitz continuous mappings
$$
\varphi_i: \mathbb R^m\to\mathbb R^n, 
$$   
such that the $m$-Hausdorff measure $H^m$ of the set $E\setminus \bigcup_{i=1}^{\infty}\varphi_i(\mathbb R^m)$
is zero. 

Let $\Omega\subset\mathbb R^n$ be a domain with $H^{n-1}$-rectifiable boundary $\partial\Omega$ and $w:\partial\Omega\to\mathbb R$ be a non-negative continuous function. We consider the weighted Lebesgue space $L^{p}_w(\partial\Om)$ with the following norm
$$
\|u\|_{L^p_w(\partial\Omega)}=\left(\int_{\partial\Om}|u(x)|^p w(x)\,ds(x)\right)^\frac{1}{p},
$$
where $ds$ is the $(n-1)$-dimensional surface measure on $\partial\Omega$.

In accordance with the outward $\gamma$-cuspidal domain $\Omega_{\gamma}$ defined by (\ref{cusp}), we define a continuous weight function $w:\partial\Omega_{\gamma}\to\mathbb R$ setting 
\begin{equation}
\label{weight}
w(x_1,...,x_{n-1},x_n)=
\begin{cases}
g(x_n),\,\,&\text{if}\,\, \sqrt{x_1^2+...+x_{n-1}^2}=g(x_n)<1,
\\
1,\,\,&\text{if}\,\, \sqrt{x_1^2+...+x_{n-1}^2}=g(x_n)\geq 1.
\end{cases}
\end{equation}

The following theorem is a direct consequence of  \cite[Theorem 4.1]{GV} and of the fact that the domains of the class $OP_{\varphi}$ considered in \cite{GV} 
are bi-Lipschitz equivalent to the outward cuspidal domains $\Omega_{\gamma}$, see \cite{KZ}.

\begin{theorem}
\label{trace}
Let $\Omega_{\gamma}\subset\mathbb R^n$ be the outward $\gamma$-cuspidal domain defined by \eqref{cusp}, 
and let $w$ be the weight defined by \eqref{weight}. 
Then the trace embedding
\[
i: W^{1,p}(\Omega_{\gamma})\hookrightarrow L^p_w(\partial\Omega_{\gamma}),\qquad 1<p<\infty,
\]
is compact.
\end{theorem}

Remark that the range of exponents $p$ is not explicitly stated in \cite{GV}. 
However, since the trace embedding in \cite{GV} is of the form $W^{1,p}\to L^p_{\varphi}$ 
(the same exponent of integrability) and the weight $\varphi$ reflects the geometry of the outward cusp, 
their argument applies for all $1<p<\infty$.

Let us recall the Sobolev embedding theorem in outward $\gamma$-cuspidal domains, see \cite{GG94,GU09} and \cite[p. 169, Theorem~8]{bur}, \cite[Theorem~5.2.2]{M}. 

\begin{theorem}\label{thmemb}
Let $\Omega_{\gamma}\subset\mathbb R^n$ be an outward $\gamma$-cuspidal domain defined by (\ref{cusp}). Then the embedding operator
$$
W^{1,p}(\Omega_{\gamma})\hookrightarrow L^{p}(\Omega_{\gamma})
$$
is compact for any $1<p<\infty$.
\end{theorem}

By using these theorems, we prove Friedrichs--Poincar\'e type inequality on cuspidal domains.

\begin{theorem}\label{FP}
Let $1<p<\infty$ and let $\Omega_{\gamma}\subset\mathbb{R}^n$ be an outward $\gamma$-cuspidal domain defined by \eqref{cusp}, and let the weight $w$ be defined by \eqref{weight}. Then there exists a constant $C=C(\Omega_{\gamma},p,w)>0$ such that 
\[
\|u\|_{L^p(\Omega_{\gamma})}\leq C\,\|\nabla u\|_{L^p(\Omega_{\gamma})},
\]
for all $u\in W^{1,p}(\Omega_{\gamma})$ satisfying the condition
\[
\int_{\partial\Omega_{\gamma}} |u|^{p-2} u \, w \, ds = 0.
\]
\end{theorem}

\begin{proof}
We argue by contradiction. Assume that the conclusion fails. Then for every $n\in\mathbb{N}$ there exists $u_n\in W^{1,p}(\Omega_{\gamma})\setminus\{0\}$ such that
\[
\int_{\partial\Omega_{\gamma}} |u_n|^{p-2} u_n \, w \, ds = 0
\quad\text{and}\quad
\|u_n\|_{L^p(\Omega_{\gamma})} > n\,\|\nabla u_n\|_{L^p(\Omega_{\gamma})}.
\]

Using homogeneity, we normalize $u_n$ so that $\|u_n\|_{L^p(\Omega_{\gamma})}=1$, setting otherwise  $v_n = u_n / \|u_n\|_{L^p(\Omega_{\gamma})}$). Then
\[
\|\nabla u_n\|_{L^p(\Omega_{\gamma})}<\frac{1}{n}\quad\Longrightarrow\quad 
\|\nabla u_n\|_{L^p(\Omega_{\gamma})}\to 0\ \text{as}\ n\to\infty,
\]
and since $\|u_n\|_{L^p(\Omega_{\gamma})}=1$ we have that $\{u_n\}$ is bounded in $W^{1,p}(\Omega_{\gamma})$.

By Theorem~\ref{thmemb} we have that 
the embedding operator
\[
W^{1,p}(\Omega_{\gamma})\hookrightarrow L^p(\Omega_{\gamma})
\]
is compact. 

Therefore, by passing to a subsequence if necessary, we obtain
\[
u_n\rightharpoonup u\ \text{in}\ W^{1,p}(\Omega_{\gamma})\quad\text{and}\quad 
u_n\to u\ \text{in}\ L^p(\Omega_{\gamma})
\]
for some $u\in W^{1,p}(\Omega_{\gamma})$. Moreover, $\|\nabla u_n\|_{L^p}\to 0$ implies 
$\nabla u_n \to 0$ strongly in $L^p(\Omega_{\gamma})$, and therefore $\nabla u=0$ a.e. in $\Omega_{\gamma}$. Since $\Omega_{\gamma}$ is connected, $u\equiv c$ is a constant.

Next, by the compactness of the trace embedding
\[
i:W^{1,p}(\Omega_{\gamma})\hookrightarrow L^p_w(\partial\Omega_{\gamma}),
\]
we also have $u_n\to u$ in $L^p_w(\partial\Omega_{\gamma})$ (passing to a subsequence if necessary). 

Passing to the limit in the boundary orthogonality, we have
\[
0=\lim_{n\to\infty}\int_{\partial\Omega_{\gamma}} |u_n|^{p-2} u_n \, w \, ds
= \int_{\partial\Omega_{\gamma}} |u|^{p-2} u \, w \, ds
= |c|^{p-2}c \int_{\partial\Omega_{\gamma}} w \, ds.
\]

Since $w\ge 0$ and $\int_{\partial\Omega_{\gamma}} w\,ds>0$, it follows that $c=0$, i.e., $u\equiv 0$.

However, the strong convergence $u_n\to u$ in $L^p(\Omega_{\gamma})$ together with 
$\|u_n\|_{L^p(\Omega_{\gamma})}=1$ implies $\|u\|_{L^p(\Omega_{\gamma})}=1$, which contradicts $u\equiv 0$. This contradiction proves the existence of $C>0$ such that
\[
\|u\|_{L^p(\Omega_{\gamma})}\le C\,\|\nabla u\|_{L^p(\Omega_{\gamma})}
\]
for all $u\in W^{1,p}(\Omega_{\gamma})$ satisfying the weighted boundary orthogonality.
\end{proof}

\medskip
\noindent
Hence, the weighted boundary orthogonality condition eliminates constant functions, and Theorem~\ref{FP} shows that on this class of functions the $W^{1,p}(\Omega_\gamma)$-norm 
is equivalent to the seminorm $\|\nabla u\|_{L^p(\Omega_\gamma)}$. In particular, the gradient term controls the full $L^p$-norm of $u$.

\section{Weighted Steklov $p$-eigenvalue problem}

Let $\Omega_{\gamma}$ be an outward $\gamma$-cuspidal domain defined by (\ref{cusp}), and let the weight function $w$ be given by (\ref{weight}). 
We consider the weighted Steklov $p$-eigenvalue problem, for $1 < p < \infty$:
\begin{equation}
\label{WePN}
\begin{cases}
-\mathrm{div}(|\nabla u|^{p-2}\nabla u) = 0 & \text{in } \Omega_{\gamma},\\
|\nabla u|^{p-2} \nabla u \cdot \nu = \lambda\, w\, |u|^{p-2}u & \text{on } \partial\Omega_{\gamma}.
\end{cases}
\end{equation}

We study this problem in the weak formulation.

\begin{definition}
\label{def-p}
We say that $(\lambda,u)\in \mathbb{R} \times (W^{1,p}(\Omega_{\gamma})\setminus\{0\})$ 
is an eigenpair of \eqref{WePN} if 
\begin{equation}
\label{skwksolN}
\int_{\Omega_{\gamma}} |\nabla u|^{p-2} \nabla u \cdot \nabla v\,dx
= \lambda \int_{\partial \Omega_{\gamma}} |u|^{p-2} u\, v\, w\,ds(x),
\end{equation}
for every $v\in W^{1,p}(\Omega_{\gamma})$.
\end{definition}

We refer to $\lambda$ as an eigenvalue and $u$ as an eigenfunction of \eqref{WePN} corresponding to the eigenvalue $\lambda$.

The equation \eqref{WePN} represents the Euler-Lagrange equation corresponding, in its weak formulation \eqref{skwksolN}, to the functional
$$
F = \|\nabla v\|_{L^p(\Omega_{\gamma})}^p,
$$
restricted to the set
$$
S = \left\{ u \in W^{1,p}(\Omega_{\gamma}) : \|u\|_{L^p_w(\partial \Omega_{\gamma})} = 1 \right\}.
$$

The following theorem provides the existence and variational characterization of the first non-trivial eigenvalue \( \lambda_p \) associated with the weighted Steklov \( p \)-eigenvalue problem, described in terms of the minimum of the Rayleigh quotient. The orthogonality condition
$$
\int_{\partial \Omega_{\gamma}} |u|^{p-2} u\, w\, ds = 0
$$
ensures that the eigenfunction is non-trivial. 

\begin{theorem}\label{minthm}
Let $\Omega_{\gamma}$ be an outward $\gamma$-cuspidal domain defined by (\ref{cusp}) and the weight $w$ be defined by (\ref{weight}). Then 
for the weighted Steklov $p$-eigenvalue problem \eqref{WePN}, $1<p<\infty$, there exists $u\in W^{1,p}(\Omega_{\gamma})\setminus\{0\}$, such that, the first non-trivial eigenvalue $\lambda_{p}$ is given by
\begin{multline}
\label{min}
\lambda_{p}=
\inf \left\{\frac{\|\nabla v\|_{L^p(\Omega_{\gamma})}^p}{\|v\|_{L^p_w(\partial \Omega_{\gamma})}^p} : v \in W^{1,p}(\Omega_{\gamma}) \setminus \{0\},
\int_{\partial \Omega_{\gamma}} |v|^{p-2}v w\,ds=0 \right\}\\
=\frac{\|\nabla u\|_{L^p(\Omega_{\gamma})}^p}{\|u\|_{L^p_w(\partial \Omega_{\gamma})}^p}.
\end{multline}
\end{theorem}

\begin{proof}
Let us consider the variational characterization of the first non-trivial eigenvalue \( \lambda_p \) given by
\[
\lambda_p = \inf \left\{ \frac{\|\nabla v\|_{L^p(\Omega_\gamma)}^p}{\|v\|_{L^p_w(\partial \Omega_\gamma)}^p} : v \in W^{1,p}(\Omega_\gamma) \setminus \{0\},\ \int_{\partial \Omega_\gamma} |v|^{p-2} v\, w\, ds = 0 \right\}.
\]
If \( \|v\|_{L^p_w(\partial \Omega_\gamma)} = 0 \), then the Rayleigh quotient is infinite, and such functions do not contribute to the infimum.

By Theorem~\ref{trace}, the trace operator
\[
i: W^{1,p}(\Omega_\gamma) \hookrightarrow L^p_w(\partial \Omega_\gamma)
\]
is compact.  

For each \( k \in \mathbb{N} \), define the perturbed functional
\[
H_{1/k}(v) = \|\nabla v\|_{L^p(\Omega_\gamma)}^p - \left( \lambda_p + \frac{1}{k} \right) \|v\|_{L^p_w(\partial \Omega_\gamma)}^p.
\]
By the definition of infimum, for each \( k \), there exists \( u_k \in W^{1,p}(\Omega_\gamma) \setminus \{0\} \) such that
\[
\int_{\partial \Omega_\gamma} |u_k|^{p-2} u_k\, w\, ds = 0,\quad \text{and} \quad H_{1/k}(u_k) < 0.
\]

Without loss of generality, normalize:
\[
\|\nabla u_k\|_{L^p(\Omega_\gamma)}^p = 1.
\]

Then by Theorem~\ref{FP} the sequence \( \{u_k\} \) is bounded in \( W^{1,p}(\Omega_\gamma) \), and by reflexivity and compactness of the trace operator, there exists \( u \in W^{1,p}(\Omega_\gamma) \) such that:
\begin{itemize}
    \item \( u_k \rightharpoonup u \) weakly in \( W^{1,p}(\Omega_\gamma) \),
    \item \( u_k \to u \) strongly in \( L^p_w(\partial \Omega_\gamma) \),
    \item \( \nabla u_k \rightharpoonup \nabla u \) weakly in \( L^p(\Omega_\gamma) \).
\end{itemize}

Since $u_k \to u$ strongly in $L^p_w(\partial\Omega_\gamma)$, and the map
$u \mapsto |u|^{p-2}u$ is continuous from $L^p$ to $L^{p/(p-1)}$, we obtain
\[
\int_{\partial \Omega_\gamma} |u|^{p-2} u\, w\, ds = \lim_{k \to \infty} \int_{\partial \Omega_\gamma} |u_k|^{p-2} u_k\, w\, ds = 0.
\]

From \( H_{1/k}(u_k) < 0 \), we get:
\begin{equation}\label{ineq}
1 - \left( \lambda_p + \frac{1}{k} \right) \|u_k\|_{L^p_w(\partial \Omega_\gamma)}^p < 0.
\end{equation}

Letting \( k \to \infty \), and using lower semicontinuity of the norm:
\[
\|\nabla u\|_{L^p(\Omega_\gamma)}^p \leq \liminf_{k \to \infty} \|\nabla u_k\|_{L^p(\Omega_\gamma)}^p = 1,
\]
and
\[
\|u\|_{L^p_w(\partial \Omega_\gamma)}^p = \lim_{k \to \infty} \|u_k\|_{L^p_w(\partial \Omega_\gamma)}^p.
\]

In particular,  passing to the limit in \eqref{ineq} as $k\to \infty$ we get 
\[
\lambda_p \|u\|_{L^p_w(\partial \Omega_\gamma)}^p \geq 1,
\]
which implies that both  \( \lambda_p > 0 \) and \( u \not\equiv 0 \).

Therefore,
\[
\lambda_p \geq \frac{\|\nabla u\|_{L^p(\Omega_\gamma)}^p}{\|u\|_{L^p_w(\partial \Omega_\gamma)}^p}.
\]

But since \( u \) satisfies the constraint and belongs to the admissible set, by definition of \( \lambda_p \), we also have:
\[
\lambda_p \leq \frac{\|\nabla u\|_{L^p(\Omega_\gamma)}^p}{\|u\|_{L^p_w(\partial \Omega_\gamma)}^p}.
\]
Hence,
\[
\lambda_p = \frac{\|\nabla u\|_{L^p(\Omega_\gamma)}^p}{\|u\|_{L^p_w(\partial \Omega_\gamma)}^p},
\]
and \( u \) is a minimizer.
\end{proof}

\vskip 0.3cm
\noindent
\textbf{Acknowledgments.}
The first named author is a member of the Gruppo Nazionale per l'Analisi Matematica, la Probabilit\`a e le loro Applicazioni (GNAMPA) of the Istituto Nazionale di Alta Matematica (INdAM) and acknowledges support from the project ``Perturbation problems and asymptotics for elliptic differential equations: variational and potential theoretic methods'' funded by the European Union -- Next Generation EU and by MUR Progetti di Ricerca di Rilevante Interesse Nazionale (PRIN) Bando 2022, grant 2022SENJZ3.

\vskip 0.5cm

\noindent {\textsf{Pier Domenico Lamberti\\
Department of Industrial Systems Engineering and Management,\\
University of Padova,\\ Padova, Italy}\\
\textsf{e-mail}: lamberti@math.unipd.it\\

\vskip 0.1cm
\noindent {\textsf{Alexander Ukhlov\\
Department of Mathematics,\\
Ben-Gurion University of the Negev,\\ Beer-Sheva,Israel}\\
\textsf{e-mail}: ukhlov@math.bgu.ac.il\\

\end{document}